\documentclass[a4paper,12pt]{article}
\usepackage{amsfonts}
\usepackage{pifont}
\usepackage{latexsym,amsmath,amssymb,cite,amsthm}
\usepackage[pdfborder={0 0 0}]{hyperref}

\usepackage{color,eucal,enumerate,mathrsfs}
\usepackage[normalem]{ulem}
\usepackage{amsmath,amssymb,epsfig,bbm}

\numberwithin{equation}{section}
\theoremstyle{plain}
\newtheorem{theorem}{Theorem}[section]

\newtheorem{proposition}[theorem]{Proposition}
\theoremstyle{definition}

\newtheorem{definition}[theorem]{Definition}
\theoremstyle{remark}

\newtheorem{remark}[theorem]{Remark}

\newcommand{\fr}{\penalty-20\null\hfill$\blacksquare$}                      

\newcommand{\lip}[1]{{\mathrm{lip}}({#1})}
\newcommand{\Lip}[1]{{\mathrm{Lip}}({#1})}

\newcommand{\supp}[1]{{\mathrm{supp}}{#1}}

\newcommand{\D}{{\rm D}}
\newcommand{\s}{{\rm S}}
\newcommand{\m}{\mathfrak {m}}
\renewcommand{\H}{{\sf H}}
\newcommand{\Ht}{\tilde{\sf H}}
\newcommand{\ppi}{{\mbox{\boldmath$\pi$}}}
\newcommand{\sfd}{{\sf d}}
\renewcommand{\d}{{\rm d}}

\newcommand{\lims}{\varlimsup}
\newcommand{\nchi}{{\raise.3ex\hbox{$\chi$}}}
\renewcommand{\D}{{\rm D}}
\newcommand{\e}{{\rm e}}
\newcommand{\eps}{\varepsilon}
\newcommand{\RCD}{{\sf RCD}}
\newcommand{\restr}[1]{\lower3pt\hbox{$|_{#1}$}}
                  
\newcommand{\probt}[1]{{\mathscr P}_2(#1)}                  
\newcommand{\E}{\sf E}
\newcommand{\R}{\mathbb R}
\newcommand{\bb}[1]{|{\mathbf D}#1|}
\newcommand{\BV}{{\rm BV}}

\title{  Independence on $p$ of weak upper gradients on $\RCD$ spaces}
\author{ Nicola Gigli\thanks{
Universit\'{e}  Pierre et Marie Curie}\and Bang-Xian Han \thanks
{ Universit\'{e} Paris Dauphine}
}

\begin{document}
\date{\today}
\maketitle

\begin{center}
\begin{minipage}{130mm}{\small
\textbf{Abstract}: We study $p$-weak gradients on $\RCD(K,\infty)$ metric measure spaces and prove that they all coincide for $p>1$.  On proper spaces, our arguments also cover the extremal situation of $\BV$ functions.
}
\\
\\
\textbf{Keywords}: Weak gradient, Sobolev space, metric measure space, optimal transport.\\
\end{minipage}
\end{center}

\tableofcontents

\section{Introduction}
There is a large literature concerning the definition of the Sobolev space $W^{1,p}(X,\sfd,\m)$ of real valued functions defined on a metric measure space $(X,\sfd,\m)$, we refer to \cite{Heinonen07} and \cite{AmbrosioGigliSavare11-3} for historical comments and a presentation of the various - mostly equivalent - approaches.

The definition of space $W^{1,p}(X,\sfd,\m)$ comes with the definition of an object playing the role of the modulus of the distributional differential. More precisely, for $f\in W^{1,p}(X,\sfd,\m)$ it is well defined a non-negative function $|\D f|_p\in L^p(X,\m)$, called minimal $p$-weak upper gradient, which, if $(X,\sfd,\m)$ is a smooth space, coincides $\m$-a.e.\ with the modulus of the distributional differential of $f$.

A key difference between the smooth and non-smooth case is that in the latter the minimal $p$-weak upper gradient may depend on $p$: say for simplicity that $\m(X)=1$, then for $p<q\in(1,\infty)$ and $f\in W^{1,q}(X)$ one always has $f\in W^{1,p}(X)$ but in general only the inequality
\begin{equation}
\label{eq:1i}
|\D f|_p\leq |\D f|_q,\qquad\m-a.e.,
\end{equation}
holds. The inequality above can be strict even on doubling spaces, see \cite{DiMarinoSpeight13} for an example and more details on the issue.

Worse than this, one might have 
\begin{equation}
\label{eq:2i}
\text{a function $f\in W^{1,p}(X)$ with $f,|\D f|_p\in L^q(X)$ such that  $f\notin W^{1,q}(X)$,}
\end{equation}
see \cite{AmbrosioGigliSavare11-3} for an example proposed by Koskela.

To have a $p$-weak upper gradients independent on $p$ is a regularity property of the metric measure space in question. For instance, as a consequence of the analysis done in  \cite{Cheeger00} one has that on doubling space supporting a 1-1 weak local Poincar\'e inequality, equality always holds in \eqref{eq:1i}. In particular, this applies to ${\sf CD}(K,N)$ spaces with  $N<\infty$.

In this note we show that on $\RCD(K,\infty)$ spaces not only \eqref{eq:1i} holds with equality, but also that the situation in \eqref{eq:2i} never occurs. The argument is based on some regularization properties of the heat flow proved in \cite{Savare13} and on the density in energy of Lipschitz functions in Sobolev spaces established in \cite{AmbrosioGigliSavare11-3}.

At least in the case of proper $\RCD(K,\infty)$ spaces, this identification extends to $\BV$ functions. The problem in  non-proper spaces  is the lack of an approximation result of $\BV$ functions with Lipschitz ones.

\bigskip

This result, beside its intrinsic usefulness in Sobolev calculus, has also the pleasant conceptual effect of somehow relieving the definition of $\RCD$ spaces from the dependence on the particular Sobolev exponent $p=2$.  Recall indeed that one of the equivalent definitions of $\RCD(K,\infty)$ space is that of a ${\sf CD}(K,\infty)$ space such that $W^{1,2}(X)$ is Hilbert or equivalently such that
\[
|\D (f+g)|_2^2+|\D(f-g)|_2^2=2\big(|\D f|_2^2+|\D g|_2^2\big),\quad\m\text{-a.e. }\qquad\forall f,g\in W^{1,2}(X).
\]
As a consequence of our result, a posteriori one could replace the minimal 2-weak upper gradients with $p$-weak upper gradients in the above.

\section{Preliminaries}

\subsection{Sobolev classes}
We assume the reader familiar with the basic concepts of analysis in metric measure spaces and recall here the definition of Sobolev class $\s^p(X)$. We fix a complete and separable space $(X,\sfd,\m)$ such that $\m$ is a non-negative Borel measure finite on bounded sets.

\begin{definition}[Test plans] Let  $\ppi$ be a Borel probability measure on $C([0,1],X)$. We say  that $\ppi$ has bounded compression provided
there exists $C =C(\ppi)>0$ such that
\[
(\e_t)_\sharp \ppi \leq C\m,\qquad\forall t \in [0,1],
\]
where $\e_t:C([0,1],X)\to X$ is the evaluation map defined by $\e_t(\gamma):=\gamma_t$ for every $\gamma\in C([0,1],X)$.

For $q\in(1,\infty)$ we say that $\ppi$ is a $q$-test plan if it has bounded compression, is concentrated on $AC^q([0,1],X)$ and
\[
\int_0^1\int |\dot{\gamma}_t|^q \,\d\ppi(\gamma)\, \d t < +\infty.
\]
\end{definition}
The notion of Sobolev function is then introduced by duality with test plans.

\begin{definition}[Sobolev classes]
Let  $p\in(1,\infty)$. The space $\s^p(X)$ is the space of all  Borel functions $f : X \to   \mathbb{R}$ for which there exists a non-negative function $G\in L^p(X)$ such that for any $q$-test plan $\ppi$ the inequality
\[
\int |f(\gamma_1)- f(\gamma_0)|\, \d\ppi(\gamma) \leq \iint_0^1 G(\gamma_s)|\dot{\gamma}_s|\, \d s\, \d\ppi(\gamma)
\]
holds, where $\frac1p+\frac1q=1$. Any such $G$ is called $p$-weak upper gradient.
\end{definition}
It is immediate to see that for $f\in\s^p(X)$ there is a unique $p$-weak upper gradient of minimal $L^p$-norm: we shall call such $G$ minimal $p$-weak upper gradient and denote it by $|\D f|_p$.

Basic important properties of minimal weak upper gradients are the \underline{locality}, i.e.:
\[
|\D f|_p=|\D g|_p,\quad\m\text{-a.e. on }\{f=g\},\qquad\forall f,g\in\s^p(X),
\]
and the \underline{lower semicontinuity}, i.e.
\[
\left.\begin{array}{ll}
(f_n)\subset \s^p(X),\\
\sup_n\||\D f|_p\|_{L^p}<\infty,\\
f_n\to f\quad\m\text{-a.e.}
\end{array}\right\}\qquad\Rightarrow\qquad\left\{\begin{array}{l}
f\in \s^p(X)\text{ and}\\
\text{for any weak limit $G$}
\text{ of $(|\D f_n|_p)$ in $L^p(X)$}\\
\text{it holds }|\D f|_p\leq G\qquad\m\text{-a.e}.
\end{array}\right.
\]
The Sobolev space $W^{1,p}(X)$ is defined as $W^{1,p}(X):= \s^p\cap L^p(X)$  endowed with the norm
\[
\|f\|^p_{W^{1,p}(X)}:=\|f\|^p_{L^p(X)}+\||\D f|_p\|^p_{L^p(X)}.
\]
By $\Lip X$ we denote the space of Lipschitz functions on $X$ and for $f\in\Lip X$ the local Lipschitz constant $\lip f:X\to[0,\infty)$ is defined as
\[
\lip f(x):=\lims_{y\to x}\frac{|f(y)-f(x)|}{\sfd(x,y)},
\]
if $x$ is not isolated, 0 otherwise.

In \cite{AmbrosioGigliSavare11-3} the following approximation property has been proved:
\begin{proposition}[Density in energy of Lipschitz functions]\label{prop.approx}
Let $(X,\sfd,\m)$ be a complete separable metric space with $\m$ being Borel non-negative and assigning finite mass to bounded sets. Let $p\in(1,\infty)$ and $f\in W^{1,p}(X)$. 

Then there exists a sequence $(f_n)\subset W^{1,p}\cap\Lip X$ of functions with bounded support converging to $f$ in $L^p(X)$ and such that $\lip{ f_n}\to |\D f|_p$ in $L^p(X)$ as $n\to\infty$.
\end{proposition}
We conclude this introduction noticing that the locality property of $p$-weak upper gradients allows for a natural definition of the space of locally Sobolev functions. By $L^p_{\rm loc}(X)$ we shall intend the space of Borel functions $G:X\to\mathbb R$ whose $p$-power is integrable on every bounded set.
\begin{definition}[The spaces $\s^p_{\rm loc}(X)$]
We say that $f\in\s^p_{\rm loc}(X)$ provided for any  Lipschitz function with bounded support $\nchi$ we have $\nchi f\in\s^p(X)$. In this case we define $|\D f|_p\in L^p_{\rm loc}(X)$ as
\[
|\D f|_p:=|\D(\nchi f)|_p,\qquad\m\text{-a.e. on }\{\nchi=1\},
\]
for every $\nchi$ as before.
\end{definition}
The role of the locality of the $p$-weak upper gradient is to ensure that the definition of $|\D f|_p$ is well posed. Also, it is not hard to check that $\s^p(X)\subset \s^p_{\rm loc}(X)$ and that a function $f\in \s^p_{\rm loc}(X)$ belongs to $\s^p(X)$ if and only if $|\D f|_p\in L^p(X)$.

For $p_1<p_2\in (1,\infty)$ and $q_1>q_2\in(1,\infty)$ such that $\frac1{p_i}+\frac1{q_i}=1$, the fact that the class of $q_1$-test plans is contained in the one of $q_2$-test plans grants that
\begin{equation}
\label{eq:easy}
\s^{p_2}_{\rm loc}(X)\subset \s^{p_1}_{\rm loc}(X)\qquad\text{and}\qquad|\D f|_{p_1}\leq |\D f|_{p_2}\quad\m\text{-a.e.}\qquad\forall f\in \s^{p_2}_{\rm loc}(X).
\end{equation}

\subsection{Heat flow on $RCD(K,\infty)$ space}
In order to keep this preliminary part as short as possible, we shall assume the reader familiar with the definition of  $\RCD(K,\infty)$ spaces and focus only on those properties they have which are relevant for our discussion. We refer  to \cite{AmbrosioGigliSavare11-2}, \cite{AmbrosioGigliMondinoRajala12} and \cite{Savare13} for the throughout discussion.

From now on we shall assume that $(X,\sfd,\m)$ is a $\RCD(K,\infty)$ space for some $K\in\R$ and that the support of $\m$ is the whole $X$. Recall that in particular we have $\m(B)<\infty$ for any bounded Borel set $B\subset X$.

In such space  the 2-Energy functional $\E:L^2(X)\to[0,\infty]$ defined as
\[
\E_2(f):=\left\{\begin{array}{ll}
\frac12 \int_X |\D f|_2^2\,\d \m,&\qquad\text{ if }f\in W^{1,2}(X),\\
+\infty,&\qquad\text{ otherwise,}
\end{array}\right.
\]
is a strongly local and regular Dirichlet form. We shall denote by $(\H_t)$ the associated linear semigroup. Then it can be seen that for every $f\in L^2(X)$ and $p\in[1,\infty)$ we have
\[
\|\H_t(f)\|_{L^p(X)}\leq \| f\|_{L^p(X)},\qquad\forall t\geq0,
\]
and thus $(\H_t)$ can, and will, be extended to a linear non-expanding semigroup on $L^p(X)$.

On the other hand, there exists a unique ${\sf EVI}_K$-gradient flow of the relative entropy functional on $(\probt X,W_2)$ which we shall denote by $(\mathcal H_t)$ and provides a one parameter semigroup of continuous linear operators on $(\probt X,W_2)$, see  \cite{AmbrosioGigliSavare11-2} and \cite{AmbrosioGigliMondinoRajala12} .

The non-trivial link between $(\H_t)$ and $(\mathcal H_t)$ is the fact that
\[
\begin{split}
&\text{for $\mu\in\probt X$ such that $\mu=f\m$ for some $f\in L^2(X)$}\\
&\text{we have $\mathcal H_t(\mu)=\H_t(f)\m$ for every  $t\geq 0$,}
\end{split}
\]
 and from the fact that $(\H_t)$ is self adjoint one can verify that for any $p\in[1,\infty)$ and every $t\geq 0$ it holds
\begin{equation}
\label{eq:version}
\H_t(f)(x)=\int f\,\d\mathcal H_t(\delta_x),\qquad\m-a.e.\ x\qquad\forall f\in L^p(X).
\end{equation}
Moreover, for $f\in L^\infty(X)$ and $t>0$ the formula
\[
\Ht_t(f)(x):=\int f\,\d\mathcal H_t(\delta_x),
\]
is well defined for any $x\in X$ producing a pointwise version of the heat flow for which the $L^\infty\to{\rm Lip}$ regularization holds:
\begin{equation}
\label{eq:linftylip}
\Lip{\Ht_t(f)}\leq \frac1{\sqrt{2I_{2K}(t)}}\|f\|_{L^\infty(X)},\qquad\forall t>0,
\end{equation}
where $I_{2K}(t):=\int_0^te^{2Ks}\,\d s$.

Th crucial regularization property of the heat flow that we shall use to identify $p$-weak gradients is the following version of the Bakry-\'Emery contraction rate, proved in \cite{Savare13}:
\begin{equation}
\label{eq:BE}
\lip{\Ht_t(f)}\leq e^{-Kt}\Ht_t(\lip f)\qquad\text{pointwise on }X,
\end{equation}
valid for every  Lipschitz function $f$ with bounded support and every $t\geq 0$.

We conclude recalling another useful regularity property of $\RCD$ spaces, this one concerning displacement interpolation of measures, see \cite{Rajala12-2} for a proof:
\begin{proposition}\label{prop:rajala}
Let $(X,\sfd,\m)$ be a $\RCD(K,\infty)$ space and $\mu,\nu$ two Borel probability measures with bounded support and such that $\mu,\nu\leq C\m$ for some $C>0$.

Then there exists a Borel probability measure $\ppi$ on $ C([0,1],X)$ such that $(\e_t)_\sharp\ppi\leq C'\m$ for every $t\in[0,1]$ for some $C'>0$ and for which the inequality
\[
\Lip{\gamma}\leq \sup_{x\in \supp(\mu)\atop y\in\supp(\nu)}\sfd(x,y),
\]
holds for every $\gamma$ in the support of $\ppi$.
\end{proposition}

\section{Proof of the main result}
The identification of $p$-weak gradients will come via a study of the regularization properties of the heat flow.

\begin{proposition}\label{prop:1}
Let $p\in(1,\infty)$, $f\in W^{1,p}(X)$ and $t\geq 0$. Then $\H_t(f)\in W^{1,p}(X)$ and 
\[
|\D \H_tf|_p^p\leq e^{-pKt}\H_t(|\D f|_p^p),\qquad\m-a.e..
\]
\end{proposition}
\begin{proof} The fact that $\H_t(f)\in L^p(X)$ follows from the fact that $f\in L^p(X)$.
By Proposition \ref{prop.approx} we  can find a sequence $(f_n)\subset W^{1,p}\cap\Lip{X}$ converging to $f$ in $L^p(X)$ and such that  $\lip{f_n}\to |\D f|_p$ in $L^p(X)$.
By the property \eqref{eq:BE}  we know that
\begin{equation}
\label{eq:1}
\lip{\Ht_t(f_n)}^p \leq e^{-pKt}\Ht_t(\lip{f_n}^p),\qquad\text{pointwise on }X.
\end{equation}
The continuity in $L^1(X)$ of the heat flow grants that 
\begin{equation}
\label{eq:2}
\Ht_t(|\lip{f_n}|^p)\to \Ht_t(|\D f|_p^p),\qquad \text{in }L^1(X),
\end{equation}
so that in particular \eqref{eq:1} grants that the sequence $(\lip{\Ht_t(f_n)})$ is bounded in $L^p(X)$. Therefore also $(|\D \Ht_t(f_n)|_p)$ is bounded in $L^p(X)$ and up to pass to a subsequence, not relabeled, we can assume that it weakly converges to some $G\in L^p(X)$. By \eqref{eq:1} and \eqref{eq:2} we have $G\leq \Ht_t(|\D f|_p^p)$ $\m$-a.e.\ while the lower semicontinuity of $p$-weak upper gradients ensures that $\H_t(f)\in\s^p(X)$ with $|\D f|_p\leq G$ $\m$-a.e.\ and the thesis follows.
\end{proof}
\begin{proposition}\label{prop:2}
Let $p\in(1,\infty)$, $f\in W^{1,p}(X)$ such that $f,|\D f|_p\in L^\infty(X)$ and $t> 0$. Then $\Ht_t(f)$ is Lipschitz and 
\[
\lip{\Ht_t(f)}\leq e^{-Kt}\sqrt[p]{\Ht_t(|\D f|_p^p)},\qquad\text{pointwise on }X.
\]
\end{proposition}
\begin{proof}
The fact that $\Ht_t(f)$ is Lipschitz follows from \eqref{eq:linftylip}. To prove the thesis, pick $x,y\in X$, $r>0$, consider the measures $\mu_{0,r}:=\m(B_r(x))^{-1}\m\restr{B_r(x)}$, $\mu_{1,r}:=\m(B_r(y))^{-1}\m\restr{B_r(y)}$ and let $\ppi$ be given by Proposition \ref{prop:rajala}. Then  we know that $(\e_s)_\sharp\ppi\leq C\m$ for some $C>0$ and every $s\in[0,1]$ and that $|\dot \gamma_s|\leq\sfd(x,y)+2r$ for $\ppi$-a.e.\ $\gamma$ and a.e.\ $s\in[0,1]$. In particular, $\ppi$ is a $q$-test plan, where $\frac 1p+\frac1q=1$, and since $\H_t(f)\in W^{1,p}(X)$ we know that
\[
\begin{split}
\Big|\int \Ht_t(f)\,\d(\mu_{1,r}- \mu_{0,r})\Big|&\leq \int |\Ht_t(f)(\gamma_1)-\Ht_t(f)(\gamma_0)|\,\d\ppi(\gamma)\\
&\leq\iint_0^1|\D \Ht_t(f)|_p(\gamma_s)|\dot\gamma_s|\,\d s\,\d\ppi(\gamma)\\
&\leq(\sfd(x,y)+2r)\sqrt[p]{\iint_0^1|\D \Ht_t(f)|_p^p(\gamma_s)\,\d s\,\d\ppi(\gamma)}\\
&\leq(\sfd(x,y)+2r)e^{-Kt}\sqrt[p]{\iint_0^1\Ht_t(|\D f|_p^p)(\gamma_s)\,\d s\,\d\ppi(\gamma)},
\end{split}
\]
having used Proposition \ref{prop:1} in the last step. Noticing that $\sfd(x,\gamma_s)\leq \sfd(x,y)+3r$ for $\ppi$-a.e.\ $\gamma$ and every $s\in[0,1]$ we deduce that
\[
\Big|\int \Ht_t(f)\,\d(\mu_{1,r}- \mu_{0,r})\Big|\leq (\sfd(x,y)+2r)e^{-Kt}\sqrt[p]{\sup_{B_{\sfd(x,y)+3r}}\Ht_t(|\D f|_p^p)},
\]
and letting $r\downarrow0$ and using the continuity of $\Ht_t(f)$ we deduce that
\[
\frac{|\Ht_t(f)(y)-\Ht_t(f)|}{\sfd(x,y)}\leq   e^{-Kt}\sqrt[p]{\sup_{B_{\sfd(x,y)+\eps}}\Ht_t(|\D f|_p^p)},\qquad\forall\eps>0.
\]
Letting $y\to x$ using the continuity of $\Ht_t(|\D f|_p^p)$ (which follows from the hypothesis $|\D f|_p\in L^\infty(X)$ and \eqref{eq:linftylip}) and the arbitrariness of $\eps>0$ we conclude.
\end{proof}

\begin{proposition}\label{prop:3}
Let $p,q\in(1,\infty)$ and $f\in \Lip X $.  Then
\[
|\D f|_q=|\D f|_p,\qquad\m-a.e..
\]
\end{proposition}
\begin{proof}Assume $p<q$. Then we already know by \eqref{eq:easy} that $|\D f|_p\leq |\D f|_q$ $\m$-a.e.. Notice that by the locality property of the weak upper gradients it is not restrictive to assume that $f$ has bounded support, so that in particular $f\in L^\infty\cap W^{1,p}(X)$. Let $t>0$ and apply Proposition \ref{prop:2} to deduce that $\Ht_t(f)\in\Lip X$ with
\[
\lip{\Ht_t(f)}^q\leq e^{-qKt}\Ht_t(|\D f|_p^p)^{\frac qp}\leq e^{-qKt}\Ht_t(|\D f|_p^q),\qquad\text{pointwise},
\]
having used Jensen's inequality and formula \eqref{eq:version}  in the last step and the fact that $|\D f|_p^q\in L^1(X)$, which follows from the fact that $f$ is Lipschitz bounded support. Since $|\D \Ht_t(f)|_q\leq \lip{\Ht_t(f)}$ $\m$-a.e., it follows that
\[
\int|\D \Ht_t(f)|_q^q\,\d \m\leq e^{-qKt}\int \Ht_t(|\D f|_p^q)\,\d\m,\qquad\forall t>0,
\]
and letting $t\downarrow 0$ and using the lower semicontinuity  of $q$-weak upper gradients we conclude that 
\[
\int|\D f|_q^q\,\d\m\leq \int |\D f|_p^q\,\d\m,
\]
which is sufficient to get the thesis.
\end{proof}
\begin{theorem}[Identification of weak upper gradients]\label{thm:main}
Let $p,q\in(1,\infty)$ and $f\in\s^p_{\rm loc}(X)$ such that $|\D f|_p\in L^q_{\rm loc}(X)$. Then $f\in\s^q_{\rm loc}(X)$ and
\[
|\D f|_q=|\D f|_p,\qquad\m-a.e..
\]
\end{theorem}
\begin{proof}
Assume that $p< q$ and notice that by \eqref{eq:easy} it is sufficient to prove that $|\D f|_p\geq |\D f|_q$ $\m$-a.e..  Replacing if necessary $f$ with $\max\{\min\{f,n\},-n\}$ and using the locality property of weak upper gradients and the arbitrariness of $n\in\mathbb N$ we can assume that $f\in L^\infty(X)$. Similarly, with a cut-off argument we reduce to the case in which $f$ has bounded support and thus in particular $|\D f|_p\in L^p\cap L^q(X)$.

With these assumptions we have $f\in W^{1,p}(X)$ and thus for $t>0$ Proposition \ref{prop:1} gives 
\[
|\D \H_tf|_p\leq e^{-Kt}\sqrt[p]{\H_t(|\D f|_p^p)},\qquad\m-a.e..
\]
Moreover, the fact that $f$ is bounded grants, by \eqref{eq:linftylip}, that $\H_t(f)$ has a Lipschitz representative $\Ht_t(f)$ and thus Proposition \ref{prop:3} gives
\[
|\D \H_tf|_q\leq e^{-Kt}\sqrt[p]{\H_t(|\D f|_p^p)},\qquad\m-a.e..
\]
Using the assumption that $|\D f|_p\in L^q(X)$ and Jensen's inequality in formula \eqref{eq:version} we deduce that $|\D \H_tf|^q_q\leq e^{-qKt} \H_t(|\D f|_p^q)$ $\m$-a.e.\ and thus
\[
\int |\D \H_tf|^q_q\,\d\m\leq e^{-qKt} \int\H_t(|\D f|_p^q)\,\d\m,\qquad\forall t>0.
\]
Letting $t\downarrow0$ and using the lower semicontinuity of $q$-weak upper gradients we conclude that
\[
\int |\D f|_q^q\,\d\m\leq \int |\D f|_p^q\,\d\m,
\]
which is sufficient to prove the thesis.
\end{proof}

\begin{remark}[The case of $BV$ functions]{\rm Recalling the notation and results of \cite{Ambrosio-DiMarino14} about $\BV$  functions and denoting by $\bb f$ the total variation measure of $f\in \BV(X)$, assume for a moment that $(X,\sfd,\m)$ is a \emph{proper} (=bounded closed sets are compact) $\RCD(K,\infty)$ space. Then the very same arguments just used allow to prove that
\begin{equation}
\label{eq:claimbv}
\begin{split}
&\text{if $f\in \BV(X)$ is such that $\bb f \ll\m$ with $\frac{\d\bb f }{\d\m}\in L^p_{\rm loc}(X)$ for some $p>1$,}\\
&\text{then $f\in\s^p_{\rm loc}(X)$ and $|\D f|_p=\frac{\d\bb f }{\d\m}$ $\m$-a.e.. }
\end{split}
\end{equation}
To see why, notice that the fact that $(X,\sfd)$ is proper and the definition of $\BV(X)$ ensures that for $f\in \BV(X)$ there is a sequence $(f_n)$ of Lipschitz functions with bounded support such that $(f_n)\to f$ in $L^1(X)$ and $\lip{f_n}\m \to \bb f $ weakly in duality with $C_c(X)$. Hence arguing as for Proposition \ref{prop:1} one gets by approximation that
\begin{equation}
\label{eq:mollbv}
f\in \BV(X)\qquad\Rightarrow\qquad \H_t(f)\in \BV(X)\qquad \bb{\H_t(f)}\leq e^{-Kt}\mathcal H_t(\bb f ).
\end{equation}
Then, using the a priori estimates on the relative entropy of $\mathcal H_t(\mu)$ in terms of the mass of $\mu$ (see \cite{AmbrosioGigliMondinoRajala12}) one obtains that for a sequence of non-negative measures $(\mu_n)$ weakly converging to some measure $\mu$ in duality with $C_b(X)$ and $t>0$, the sequence $n\mapsto g_n:=\frac{\d\mathcal H_t(\mu_n)}{\d\m}$ converges to $g:=\frac{\d\mathcal H_t(\mu)}{\d\m}$ weakly in duality with $L^\infty(X)$. Therefore, for $\ppi$ as in the proof of Proposition \ref{prop:2} and $(f_n)\subset \Lip X$ converging to $f\in \BV(X)$ and so that $\lip{f_n}\m\to \bb f $ weakly in duality with $C_b(X)$, we can pass to the limit in the inequality
\[
\begin{split}
\int|\H_t(f_n)(\gamma_1)-\H_t(f_n)(\gamma_0)|\,\d\ppi(\gamma)&\leq\iint_0^1\lip{\H_t(f_n)}(\gamma_t)|\dot\gamma_t|\,\d t\,\d\ppi(\gamma)\\
&\leq e^{-Kt}\iint_0^1\H_t(\lip{f_n})(\gamma_t)|\dot\gamma_t|\,\d t\,\d\ppi(\gamma),
\end{split}
\]
to deduce that
\[
\int|\H_t(f)(\gamma_1)-\H_t(f)(\gamma_0)|\,\d\ppi(\gamma)\leq e^{-Kt}\iint_0^1\frac{\d\mathcal H_t(|Df|_w)}{\d\m}(\gamma_t)|\dot\gamma_t|\,\d t\,\d\ppi(\gamma).
\]
In particular, arguing as in the proof of Proposition \ref{prop:2} we get that
\begin{equation}
\label{eq:bvlip}
f\in BV\cap L^\infty(X),\ |Df|_w\leq C\m\qquad\Rightarrow\qquad\lip{\H_t(f)}\leq e^{-Kt}\mathcal H_t\Big(\frac{\d|D f|_w}{\d\m}\Big).
\end{equation}
Then following the same lines of thought of Proposition \ref{prop:3}  and Theorem \ref{thm:main} the claim \eqref{eq:claimbv} follows.

Notice also that from \eqref{eq:mollbv} and with a truncation and mollification argument we deduce that
\begin{quote}
for $f\in \BV(X)$ with $\bb f\ll\m$ there is a sequence $(f_n)\subset \Lip X$ such that  $f_n\to f$ and  $\lip{f_n}\to\frac{\d\bb f }{\d\m}$ strongly in $L^1(X)$ as $n\to\infty$.
\end{quote}
In particular, the three notions of space $W^{1,1}(X)$ discussed  in \cite{Ambrosio-DiMarino14} all coincide.

All this if the space is proper. It is very natural to expect that the same results hold even without this further assumption, but in the general case it seems necessary to define $\BV$ functions taking limits of locally Lipschitz functions, rather than Lipschitz ones (see the proof of Lemma 5.2 in \cite{Ambrosio-DiMarino14}). The problem then consists in the fact that the property \eqref{eq:BE} is not available for locally Lipschitz functions with local Lipschitz constant in $L^1$. 
}\fr\end{remark}

\def\cprime{$'$} \def\cprime{$'$}

\end{document}